\documentclass[a4paper,10pt]{amsart}
\usepackage{a4wide}
\usepackage[latin1]{inputenc}
\usepackage[T1]{fontenc}
\usepackage{amsmath,amssymb,amsthm,stmaryrd}
\usepackage{mathrsfs}
\usepackage{esint}

\newcommand{\ud}[0]{\,\mathrm{d}}

\newcommand{\eps}[0]{\varepsilon}

\newcommand{\abs}[1]{|#1|}

\newcommand{\Norm}[2]{\|#1\|_{#2}}

\newcommand{\ave}[1]{\langle #1\rangle}

\newcommand{\BMO}[0]{\operatorname{BMO}}
\newcommand{\RBMO}[0]{\operatorname{RBMO}}

\newcommand{\loc}[0]{\operatorname{loc}}

\newcommand{\R}{\mathbb{R}}

\newcommand{\N}{\mathbb{N}}
\newcommand{\Z}{\mathbb{Z}}

\newcommand{\ontop}[2]{\begin{smallmatrix} #1 \\ #2 \end{smallmatrix}}

\swapnumbers
\numberwithin{equation}{section}

\makeatletter
  \let\c@equation\c@subsection
\makeatother

\theoremstyle{plain}

\newtheorem{lemma}[subsection]{Lemma}

\newtheorem{corollary}[subsection]{Corollary}
\newtheorem{proposition}[subsection]{Proposition}

\theoremstyle{definition}
\newtheorem{definition}[subsection]{Definition}

\theoremstyle{remark}
\newtheorem{remark}[subsection]{Remark}

\title[Non-homogeneous analysis on metric spaces]{A framework for non-homogeneous analysis on metric spaces, and the RBMO space of Tolsa}
\author{Tuomas Hyt\"onen}
\address{Department of Mathematics and Statistics\\ University of Helsinki\\ Gustaf H\"allstr\"omin katu 2b\\ FI-00014 Helsinki\\ Finland}
\email{tuomas.hytonen@helsinki.fi}

\begin{document}

\maketitle

\begin{abstract}
A new class of metric measure spaces is introduced and studied. This class generalises the well-established doubling metric measure spaces as well as the spaces $(\R^n,\mu)$ with $\mu(B(x,r))\leq Cr^d$, in which non-doubling harmonic analysis has recently been developed. It seems to be a promising framework for an abstract extension of this theory. Tolsa's space of regularised BMO functions is defined in this new setting, and the John--Nirenberg inequality is proven.\\

\end{abstract}

\section{Introduction}

Spaces of homogeneous type --- (quasi-)metric spaces equipped with a so-called doubling measure --- were introduced by Coifman and Weiss \cite{CW} as a general framework in which several results from real and harmonic analysis on Euclidean spaces have their natural extension. This applies in particular to the Calder\'on--Zygmund theory of singular integrals in $L^p$, $1<p<\infty$, and in the appropriate end-point spaces of this scale. If one is willing to assume somewhat more (in particular, versions of the Poincar\'e inequality), then one can even incorporate results dealing  with first order differential calculus in a suitable generalised sense. 
These last mentioned developments, in the setting of homogeneous spaces with some additional structure, are in the core of what is now commonly referred to as analysis on metric spaces (cf. \cite{Heinonen:book,Heinonen:survey}). 

Meanwhile, recent developments in the Calder\'on--Zygmund theory (which one might think of as ``zeroth order calculus'', as only integrability and no differentiability of the functions on which one operates is considered) have shown that a number of interesting problems cannot be, and need not be, embedded into the homogeneous framework. A prime example is the question of $L^p$ boundedness of the Cauchy integral operator with respect to a measure without the doubling property \cite{NTV:Cauchy,Tolsa:Cauchy}, and more generally the new generation of Calder\'on--Zygmund operators modelled after it. Also the end-point spaces of the $L^p$ scale, and the related mapping properties of operators, have been successfully investigated in non-homogeneous situations. Some highlights of this theory, each building on the previous one, are the introduction of the regularised BMO space by Tolsa \cite{Tolsa:RBMO}, the proof of a non-homogeneous $Tb$ theorem by Nazarov, Treil and Volberg \cite{NTV:Tb}, and the solution of the Painlev\'e problem, again by Tolsa \cite{Tolsa:Painleve}.

Notwithstanding these impressive achievements, one should note that the non-homogeneous Calder\'on--Zygmund theory, as developed in most of the papers on the subject, is not in all respects a generalisation of the corresponding homogeneous theory. In fact, the typical setting there consists of $\R^n$ with a measure $\mu$ having the upper power bound $\mu(B(x,r))\leq Cr^d$ for some $d\in(0,n]$. So, first of all, it is not analysis on metric spaces, and even on $\R^n$, it deals with a class of measures which is different from, not more general than, the doubling measures. (A notable exception to the Euclidean restriction consists of the papers of Garc\'ia-Cuerva and Gatto \cite{GCG:fractional,GCG:CZO} and Gatto \cite{Gatto:09}, where some results concerning fractional, singular and hypersingular integrals on Lipschitz spaces are obtained in abstract metric spaces. A very recent work of Bramanti \cite{Bramanti} even deals with $L^p$ boundedness.
 However, when Garc\'ia-Cuerva and Gatto come to a limiting case of their results concerning the regularised BMO space, both \cite{GCG:fractional,GCG:CZO} restrict themselves to the original set-up of Tolsa on $\R^n$, and Bramanti's work completely bypasses the BMO aspects, which would be essential for obtaining $Tb$ theorems in full generality.)

The starting point of the present investigation is the following assertion by Nazarov, Treil and Volberg \cite[p.~153]{NTV:Tb}: ``The theory [of Calder\'on--Zygmund operators and $Tb$ theorems on non-homogeneous spaces] can be developed in an abstract metric space with a measure, but we will consider the interesting case for applications when our space is just a subset of $\R^N$.'' In this paper, I do not yet attempt a comprehensive justification of their claim; however, I propose a precise formulation of an abstract framework in which such a theory could be hoped for, and I take the first steps in its development by defining and investigating the regularised BMO space of Tolsa in this new setting. Garc\'ia-Cuerva and Gatto \cite{GCG:fractional} already pointed out that Tolsa's ``definition makes sense, in principle, in our general setting,'' but they did not comment on the possibility of also extending some of Tolsa's theorems concerning this space to metric spaces, which will be done here. This should also open the door for developing the results of \cite{GCG:fractional} in this wider generality.

The proposed framework is sufficiently general to include in a natural way both the abstract doubling metric measure spaces and the power-bounded measures on $\R^n$ which have been in the centre of much of today's non-doubling theory. In this sense it seems to be the ``right'' one. It does not, however, cover some other situations where non-doubling Calder\'on--Zygmund theory has been developed, such as the Gaussian measure spaces on $\R^n$ investigated by Mauceri and Meda~\cite{MaMe}.

The plan of this paper is as follows. The general framework for non-homogeneous analysis on metric spaces is set up in Section~\ref{sec:framework}. In Section~\ref{sec:differentiation}, a version of Lebesgue's differentiation theorem in this setting is obtained. Section~\ref{sec:RBMO} introduces the space $\RBMO(\mu)$ and Section~\ref{sec:RBMOaux} is concerned with some basic lemmas related to this space. The main result is the John--Nirenberg inequality proven in the final Section~\ref{sec:JN}.

As it turns out, it is possible to reasonably closely follow the original Euclidean arguments due to Tolsa \cite{Tolsa:RBMO} and reworked by Nazarov, Treil and Volberg \cite{NTV:Tb}, whose approach has been used as the primary model of the present one. There is, however, at least one place where a slight departure from their reasoning was necessary. In proving the John--Nirenberg inequality for the regularised BMO functions, both \cite{NTV:Tb} and \cite{Tolsa:RBMO} resort to the Besicovitch covering theorem, which is an essentially Euclidean device. In the abstract setting of present interest, there is not much more than the ``basic covering theorem'' \cite[Theorem 1.2]{Heinonen:book} available, and one has to survive with this weaker tool.

I will restrict myself to a metric space, although the results of this paper could be developed also with a quasi-metric satisfying only the weak triangle inequality $d(x,y)\leq K[d(x,z)+d(z,y)]$ involving a constant $K\geq 1$. The interested reader will easily realise how to modify the statements and proofs where necessary. They will not become more difficult, only somewhat more annoying.

\subsection{Notation} Following the usual practise in the area, a \emph{ball} indicates an open set $B=B(x,r)=\{y\in X:d(y,x)<r\}$ which is equipped with a fixed centre $x\in X$ and radius $r>0$, even though these need not be uniquely determined by $B$ as a set in general. Sometimes, the centre and radius of $B$ are denoted by $c_B$ and $r_B$, or by $c(B)$ and $r(B)$, depending on what seems convenient in a particular place. For $\alpha>0$ and $B=B(x,r)$, the notation $\alpha B:=B(x,\alpha r)$ stands for the concentric dilation of $B$.

Given a Borel measure $\mu$ on $(X,d)$, \emph{local integrability} will refer to integrability over all \emph{bounded} subsets of $X$. (Compactness will not play a r\^ole in the arguments, and no reference to it will be made.) All measures to be considered will be finite on such sets. For a function $f\in L^1_{\loc}(X,\mu)$, its average in a ball $B$ is denoted by
\begin{equation*}
  \ave{f}_B:=\fint_B f\ud\mu:=\frac{1}{\mu(B)}\int_B f\ud\mu.
\end{equation*}
The notation $f_B$, which is sometimes also used for $\ave{f}_B$, in this paper only indicates some number, which is related to $f$ and $B$ but need not be the same as the mentioned average value. A number of constants will be given special names, but otherwise the letter $C$ stands for a constant which only depends on the parameters of the space and never on the functions under consideration, but otherwise its value may be different at different occurences.

\subsubsection*{Acknowledgments}
A first version of the results of this paper was presented in the Analysis Seminar of Helsinki University of Technology on April 7, 2009. Juha Kinnunen encouraged me to write them down. The research was supported by the Academy of Finland, project 114374.

\section{Different notions of doubling}\label{sec:framework}

As it turns out, the ``non-doubling'' theory is in a sense more doubling than the doubling theory, in that the single classical doubling hypothesis will be replaced by a couple of other ones. In order to avoid confusion between the different hypotheses, slightly pedantic language will be employed.

\begin{definition}
A metric measure space $(X,d,\mu)$ is said to by \emph{measure doubling} if $\mu$ is a Borel measure on $X$ and there exists a constant $C_{\mu}$ such that
\begin{equation}\label{eq:measDoubling}
  0<\mu(2B)\leq C_{\mu}\mu(B)<\infty
\end{equation}
for all balls $B\subseteq X$.
\end{definition}

Measure doubling will not be assumed in this paper. The first condition replacing \eqref{eq:measDoubling} in the present investigation is also well known in analysis on metric spaces. Some relevant facts will be collected first for easy reference.

\begin{lemma}\label{lem:geomDoubling}
For a metric space $(X,d)$, the following conditions are equivalent.
\begin{enumerate}
  \item\label{it:halfCover} Any ball $B(x,r)\subseteq X$ can be covered by at most $N$ balls $B(x_i,r/2)$.
  \item\label{it:deltaCover} For every $\delta\in(0,1)$, any ball $B(x,r)\subseteq X$ can be covered by at most $N\delta^{-n}$ balls $B(x_i,\delta r)$.
  \item\label{it:deltaDisj} For every $\delta\in(0,1)$, any ball $B(x,r)\subseteq X$ can contain at most $N\delta^{-n}$ centres $y_j$ of disjoint balls $B(y_j,\delta r)$.
  \item\label{it:quarterDisj} Any ball $B(x,r)\subseteq X$ can contain at most $N$ centres $y_j$ of disjoint balls $B(y_j,r/4)$.
\end{enumerate}
\end{lemma}

\begin{proof}
$\eqref{it:halfCover}\Rightarrow\eqref{it:deltaCover}$. Let $2^{-k}\leq\delta<2^{1-k}$ with $k\in\Z_+$. By iterating \eqref{it:halfCover}, it follows that $B(x,r)$ can be covered by at most $N^k$ balls $B(x_i,2^{-k}r)\subseteq B(x_i,\delta r)$, and here $N^k=2^{k\log_2 N}\leq(2\delta^{-1})^{\log_2 N}=N\delta^{-\log_2 N}$.

$\eqref{it:deltaCover}\Rightarrow\eqref{it:deltaDisj}$. Suppose that $y_j\in B(x,r)$, $j\in J$, are centres of disjoint balls $B(y_j,\delta r)$, and choose a cover of $B(x,r)$ consisting of balls $B(x_i,\delta r)$, $i\in I$, where $\abs{I}\leq N\delta^{-n}$. Then every $y_j$ belongs to some $B(x_i,\delta r)$, and no two $y_j\neq y_k$ can belong to the same $B(x_i,\delta r)$, for otherwise $x_i\in B(y_j,\delta r)\cap B(y_k,\delta r)=\varnothing$. Thus $\abs{J}\leq\abs{I}\leq N\delta^{-n}$.

$\eqref{it:deltaDisj}\Rightarrow\eqref{it:quarterDisj}$ is obvious.

$\eqref{it:quarterDisj}\Rightarrow\eqref{it:halfCover}$. Keep selecting disjoint balls $B(y_j,r/4)$ with $y_j\in B(x,r)$ as long as it is possible; the process will terminate after at most $N$ steps by assumption. Then every $y\in B(x,r)$ belongs to some $B(y_j,r/2)$, for otherwise the ball $B(y,r/4)$ could still have been chosen.
\end{proof}

\begin{definition}
A metric space $(X,d)$ is called \emph{geometrically doubling} if the equivalent conditions of Lemma~\ref{lem:geomDoubling} are satisfied.
\end{definition}

It is well known that measure doubling implies geometrical doubling; indeed, it is one of the first things pointed out by Coifman and Weiss in their discussion of spaces of homogeneous type \cite[p. 67]{CW}. Conversely, if $(X,d)$ is a complete, geometrically doubling metric space, then there exists a Borel measure $\mu$ on $X$ such that $(X,d,\mu)$ is measure doubling \cite{LS,VK,Wu}. However, the point of view taken in the present investigation is that the measure $\mu$ is given by a particular problem, and not something that one is free to choose or construct. So even if there exist some doubling measures on the metric space of interest, one might still have to work with a non-doubling one. This is, for example, manifestly the case in the analysis of non-doubling measures on $\R^n$.

\begin{lemma}\label{lem:countable}
In a geometrically doubling metric space, any disjoint collection of balls is at most countable.
\end{lemma}

\begin{proof}
Let $x_0\in X$ be a fixed reference point. By \eqref{it:deltaDisj} of Lemma~\ref{lem:geomDoubling}, any ball $B(x_0,k)$ contains at most finitely many centres of disjoint balls of radius bigger than a given $j^{-1}$. Since every ball has its centre in some $B(x_0,k)$ and radius bigger than some $j^{-1}$, where $j,k\in\Z_+$, the conclusion follows.
\end{proof}

It is finally time to specify the class of measures to be investigated.

\begin{definition}
A metric measure space $(X,d,\mu)$ is said to be \emph{upper doubling} if $\mu$ is a Borel measure on $X$ and there exists a \emph{dominating function} $\lambda:X\times\R_+\to\R_+$ and a constant $C_{\lambda}$ such that
\begin{equation*}
  r\mapsto\lambda(x,r)\text{ is non-decreasing},
\end{equation*}
\begin{equation*}
  \lambda(x,2r)\leq C_{\lambda}\lambda(x,r),
\end{equation*}
\begin{equation*}
  \mu(B(x,r))\leq\lambda(x,r)
\end{equation*}
for all $x\in X$ and $r>0$.
\end{definition}

For some time I thought that one would also need to assume something like $\lambda(y,r)\leq C\lambda(x,r)$ when $d(y,x)\leq r$ but this turned out, at least for the present purposes, to be unnecessary. The point is that, whether or not the mentioned domination holds, one can already estimate $\mu(B(y,r))$ by $\lambda(x,r)$ for $d(x,y)\leq r$ by the existing assumptions; indeed,
\begin{equation*}
  \mu(B(y,r))\leq\mu(B(x,2r))\leq\lambda(x,2r)\leq C_{\lambda}\lambda(x,r).
\end{equation*}

It is seen at once that measure doubling is a special case of upper doubling, where one can take the dominating function to be $\lambda(x,r)=\mu(B(x,r))$. On the other hand, much of today's non-doubling theory has been developed for a measure $\mu$ on $\R^n$ which is upper doubling with the dominating function $\lambda(x,r)=Cr^d$.

In contrast to measure doubling, both geometrical doubling and upper doubling are stable under restriction to subsets: if $(X,d,\mu)$ is geometrically doubling or upper doubling and $Y\subset X$, then so is $(Y,d|_Y,\mu|_Y)$. For geometrical doubling, this is most easily seen by using condition \eqref{it:deltaDisj} or \eqref{it:quarterDisj} of Lemma~\ref{lem:geomDoubling}. As for upper doubling, it is clear that the restriction $\lambda|_{Y\times\R_+}$ of the original dominating function works. 

Measure doubling, on the other hand, already fails for the Lebesgue measure on subsets of $\R^n$ with appropriate cusps. In a closed subset of a complete, measure doubling space, there always exists a doubling measure, but it may be different from the restriction of the original measure of interest to this subset. Such a restriction is still upper doubling, though.

\section{Doubling balls and differentiation}\label{sec:differentiation}

Even if the measure doubling condition~\eqref{eq:measDoubling} is not assumed uniformly for all balls, it makes sense to ask whether such an inequality is true for a given particular ball or not.

\begin{definition}
For $\alpha,\beta>1$, a ball $B\subseteq X$ is called $(\alpha,\beta)$-doubling if $\mu(\alpha B)\leq\beta\mu(B)$.
\end{definition}

Of course, measure doubling is precisely the requirement that every ball $B$ is $(2,\beta)$-doubling for some fixed $\beta$. But even the weaker notions of geometrical doubling and upper doubling ensure the abundance of both large and small doubling balls.

\begin{lemma}\label{lem:largeDoubling}
Let the metric measure space $(X,d,\mu)$ be upper doubling and $\beta>C_{\lambda}^{\log_2\alpha}$. Then for every ball $B\subseteq X$ there exists $j\in\N$ such that $\alpha^j B$ is $(\alpha,\beta)$-doubling.
\end{lemma}

\begin{proof}
Assume contrary to the claim that none of the balls $\alpha^j B$, $j\in\N$, is $(\alpha,\beta)$-doubling, i.e., $\mu(\alpha^{j+1}B)>\beta\mu(\alpha^j B)$ for all $j\in\N$. It follows that
\begin{equation*}
\begin{split}
  \mu(B)
  &\leq\beta^{-1}\mu(\alpha B)\leq\ldots\leq\beta^{-j}\mu(\alpha^j B) \\
  &\leq\beta^{-j}\lambda(c_B,\alpha^j r_B)
  \leq\beta^{-j} C_{\lambda}^{j\log_2\alpha+1}\lambda(c_B,r_B)
  =C_{\lambda}\Big(\frac{C_{\lambda}^{\log_2\alpha}}{\beta}\Big)^j\lambda(c_B,r_B)
  \underset{j\to\infty}{\longrightarrow} 0.
\end{split}
\end{equation*}
Hence $\mu(B)=0$. But the same argument also holds with $\alpha B$ in place of $B$, leading to $\mu(\alpha B)=0$. Then $B$ is $(\alpha,\beta)$-doubling after all, which is a contradiction.
\end{proof}

\begin{lemma}\label{lem:smallDoubling}
Let $(X,d)$ be geometrically doubling and $\beta>\alpha^n$, where $n$ is as in condition \eqref{it:deltaDisj} of Lemma~\ref{lem:geomDoubling}. If $\mu$ is a Borel measure on $X$ which is finite on bounded sets, then for $\mu$-a.e. $x\in X$ there are arbitrarily small $(\alpha,\beta)$-doubling balls centred at $x$. In fact, their radius may be chosen to be of the form $\alpha^{-j}r$, $j\in\N$, for any preassigned number $r>0$.
\end{lemma}

\begin{proof}
Consider a fixed ball $B=B(x_0,r)$. It suffices to prove the claim for $\mu$-a.e. $x\in B$.

For $x\in X$ and $k\in\N$, denote $B_x^k:=B(x,\alpha^{-k}r)$. The point $x$ is called \emph{$k$-bad} if none of the balls $\alpha^j B_x^k$, $j=0,\ldots,k$, is $(\alpha,\beta)$-doubling. Note that $\alpha^k B_x^k=B(x,r)\subseteq 3B$, so for every $k$-bad point $x$ there holds
\begin{equation*}
  \mu(B_x^k)\leq\beta^{-k}\mu(\alpha^k B_x^k)\leq\beta^{-k}\mu(3B).
\end{equation*}

Among the $k$-bad points, choose a maximal $\alpha^{-k}r$-separated family $Y$. Hence the balls $B_y^k$, $y\in Y$, cover all the bad points. On the other hand, the balls $2^{-1}B_y^k=B(y,2^{-1}\alpha^{-k}r)$ are disjoint with their centres contained in $B=B(x,r)$, and hence there are at most $N(2^{-1}\alpha^{-k})^{-n}=N2^n\alpha^{kn}$ of them. Thus
\begin{equation*}
\begin{split}
  \mu(\{x\in B\ k\text{-bad}\})
  &\leq \mu\Big(\bigcup_{y\in Y}B_y^k\Big)
  \leq \sum_{y\in Y}\mu(B_y^k) \\
  &\leq \sum_{y\in Y}\beta^{-k}\mu(3B)
  \leq N2^n\mu(3B)\Big(\frac{\alpha^n}{\beta}\Big)^k
  \underset{k\to\infty}{\longrightarrow} 0.
\end{split}
\end{equation*}
Hence only a zero-set of points can be $k$-bad for all $k\in\N$, and this is precisely what was claimed.
\end{proof}

\begin{proposition}\label{prop:contDense}
Let $(X,d)$ be a geometrically doubling metric space and $\mu$ a Borel measure on $X$ which is finite on bounded sets. Then continuous, boundedly supported functions are dense in $L^p(X,\mu)$ for $p\in[1,\infty)$.
\end{proposition}

\begin{proof}
It suffices to approximate the indicator of a Borel set $E$ of finite measure in the $L^p$ norm by a continuous, boundedly supported function. Since $\mu(E)=\lim_{r\to\infty}\mu(E\cap B(x_0,r))$, there is no loss of generality in taking $E$ to be bounded. By a general result concerning Borel measures on metric spaces \cite[Theorem 2.2.2]{Federer}, there is a closed set $F\subseteq E$ and an open set $\mathscr{O}\supseteq E$ so that $\mu(\mathscr{O}\setminus F)<\eps$. Since $\mathscr{O}$ may be replaced by $\mathscr{O}\cap B$, where $B$ is any ball containing $E$, one can take $\mathscr{O}$ to be bounded.

Let $\beta>6^n$, as required in Lemma~\ref{lem:smallDoubling}. For each $x\in F$, choose a $(6,\beta)$-doubling ball $B_x$ of radius $r_x\leq 1$ centred at $x$ with $6B_x\subseteq\mathscr{O}$. By the basic covering theorem \cite[Theorem 1.2]{Heinonen:book}, extract a disjoint (hence countable by Lemma~\ref{lem:countable}) subcollection $B^i=B_{x_i}$ such that $F\subseteq\bigcup_{i=1}^{\infty}5B^i$. Since
\begin{equation*}
  \sum_{i=1}^{\infty}\mu(\overline{5B^i})
  \leq\sum_{i=1}^{\infty}\mu(6B^i)\leq\beta\sum_{i=1}^{\infty}\mu(B^i)
  =\beta\mu\Big(\bigcup_{i=1}^{\infty}B^i\Big)
  \leq\beta\mu(\mathscr{O})<\infty,
\end{equation*}
it follows that $\lim_{j\to\infty}\mu(\bigcup_{i>j}\overline{5B^i})=0$, thus $\mu(F)=\lim_{j\to\infty}\mu(F\cap\bigcup_{i=1}^j\overline{5B^i})$, and hence $F$ can be replaced by the closed set $F\cap\bigcup_{i=1}^j\overline{5B^i}$ for some large $j\in\N$. Since $6B^i\subseteq\mathscr{O}$, it follows that $d(\overline{5B^i},\mathscr{O}^c)\geq r(B^i)$, so the new set $F$ satisfies $d(F,\mathscr{O}^c)>0$. Thus the function
\begin{equation*}
  \varphi(x):=\frac{d(x,\mathscr{O}^c)}{d(x,\mathscr{O}^c)+d(x,F)}
\end{equation*}
is continuous as the quotient of continuous functions, with denominator bounded away from zero, and satisfies $1_F\leq\varphi\leq 1_{\mathscr{O}}$, where $\mathscr{O}$ is a bounded set. Hence $\abs{1_E-\varphi}\leq 1_{\mathscr{O}\setminus F}$, and thus $\Norm{1_E-\varphi}{p}^p\leq\mu(\mathscr{O}\setminus F)<\eps$.
\end{proof}

Let us consider the following variant of the Hardy--Littlewood maximal function:
\begin{equation*}
  \tilde{M}f(x):=\sup_{B\owns x}\frac{1}{\mu(5B)}\int_B\abs{f}\ud\mu,
\end{equation*}
where the supremum is over all balls $B$ containing $x$. For any $\mu$-measurable function $f$, the maximal function $\tilde{M}f$ is lower semi-continuous, hence Borel measurable.

\begin{proposition}\label{prop:maximal}
If $(X,d)$ is geometrically doubling, and $\mu$ is a Borel measure on $X$ which is finite on bounded sets, the maximal operator maps $\tilde{M}:L^1(X,\mu)\to L^{1,\infty}(X,\mu)$ boundedly.
\end{proposition}

\begin{proof}
Consider first a modified maximal operator $\tilde{M}_R$, where the supremum is restricted to balls of radius at most $R$. Then $\tilde{M}_R f$, too, is lower semi-continuous. For every $x\in\{\tilde{M}_Rf>t\}$, there exists a ball $B_x$ of radius at most $R$ such that $\mu(5B_x)^{-1}\int_{B_x}\abs{f}\ud\mu>t$. In particular, the balls $B_x$ of uniformly bounded radius cover the set $\{\tilde{M}_Rf>t\}$. By the basic covering theorem \cite[Theorem 1.2]{Heinonen:book}, among these balls one can pick a disjoint (hence countable by Lemma~\ref{lem:countable}) subcollection $B^i$, $i\in I$, so that the balls $5B^i$ still cover $\{\tilde{M}_R f>t\}$. Thus
\begin{equation*}
  \mu(\tilde{M}_R f>t)\leq\mu\Big(\bigcup_{i\in I}5B^i\Big)
  \leq\sum_{i\in I}\mu(5 B^i)
  \leq\frac{1}{t}\sum_{i\in I}\int_{B^i}\abs{f}\ud\mu
  \leq\frac{1}{t}\Norm{f}{1}.
\end{equation*}
Since $\tilde{M}_R f\uparrow\tilde{M}f$, the result follows from dominated convergence.
\end{proof}

\begin{corollary}\label{cor:differentiation}
Let $(X,d)$ be a geometrically doubling metric space and $\mu$ be a Borel measure on $X$ which is finite on bounded sets. Let $\beta>5^n$, where $n$ is as in condition \eqref{it:deltaDisj} of Lemma~\ref{lem:geomDoubling}. Then for all $f\in L^1_{\loc}(X,\mu)$ and $\mu$-a.e. $x\in X$,
\begin{equation*}
  f(x)=\lim_{\ontop{B\downarrow x}{(5,\beta)\textup{-doubling}}}\fint_B f\ud\mu,
\end{equation*}
where the limit is along the decreasing family of all $(5,\beta)$-doubling balls containing $x$, ordered by set inclusion.
\end{corollary}

\begin{proof}
By Lemma~\ref{lem:smallDoubling}, there exist arbitrarily small $(5,\beta)$-doubling balls containing $x$, so that the limit makes sense for $\mu$-a.e. $x\in X$. By a standard localisation, it suffices to consider $f\in L^1(X,\mu)$. The assertion is furthermore clear for continuous boundedly supported functions, which are dense in $L^1(X,\mu)$ by Proposition~\ref{prop:contDense}. For $f\in L^1(X,\mu)$ and a continous boundedly supported $g$,
\begin{equation*}
\begin{split}
  &\limsup_{\ontop{B\downarrow x}{(5,\beta)\textup{-doubling}}}\fint_B\abs{f(y)-f(x)}\ud\mu(y) \\
  &\leq\sup_{B\owns x}\frac{\beta}{\mu(5B)}\int_B\abs{f(y)-g(y)}\ud\mu(y)
    +\abs{g(x)-f(x)} = \beta\tilde{M}(f-g)(x)+\abs{g(x)-f(x)}.
\end{split}
\end{equation*}
By Proposition~\ref{prop:maximal}, the function on the right exceeds a given $\eps>0$ in a set of $\mu$-measure at most $C\eps^{-1}\Norm{g-f}{1}$. Since this can be made arbitrarily small by the choice of $g$, the left side must vanish $\mu$-a.e.
\end{proof}

\section{The RBMO space of Tolsa}\label{sec:RBMO}

Tolsa's \cite{Tolsa:RBMO} definition of $\RBMO(\mu)$, where $\mu$ is a Borel measure on $\R^n$ with $\mu(B)\leq Cr_B^d$, is generalised to the present setting in a straightforward way. As mentioned before, the present discussion is more closely modelled after that of Nazarov, Treil and Volberg \cite{NTV:Tb}. It is assumed throughout the section that $(X,d,\mu)$ is a geometrically doubling and upper doubling metric measure space.

\begin{definition}
Fix a parameter $\varrho>1$. A function $f\in L^1_{\loc}(\mu)$ is said to be in the space $\RBMO(\mu)$ (regularised bounded mean oscillation) if there exists a number $A$, and for every ball $B$, a number $f_B$ (which is \emph{not} required to be the average value $\ave{f}_B:=\fint_B f\ud\mu$), such that
\begin{equation}\label{eq:BMOrho}
  \frac{1}{\mu(\varrho B)}\int_B\abs{f-f_B}\ud\mu\leq A,
\end{equation}
and, whenever $B\subset B_1$ are two balls,
\begin{equation}\label{eq:regularity}
  \abs{f_B-f_{B_1}}\leq A\Big\{1+\int_{2B_1\setminus B}\frac{\ud\mu(x)}{\lambda(c_B,d(x,c_B))}\Big\}.
\end{equation}
The infimum of the admissible constants $A$ is denoted by $\Norm{f}{\RBMO}$.
\end{definition}

The condition \eqref{eq:regularity} depends on the choice of the function $\lambda$. However, it is understood that the considered metric measure space $(X,d,\mu)$ with the upper doubling property is equipped with a fixed dominating function $\lambda$, also used above, so that there is no need to indicate such dependence explicitly in the notation. The dependence on the parameter $\varrho>1$ is only implicit, as will be shown in Lemma~\ref{lem:varrho}

\begin{remark}\label{rem:regUpperBound}
The upper bound in \eqref{eq:regularity} can be further dominated as follows:
\begin{equation*}
\begin{split}
  &\int_{2B_1\setminus B}\frac{\ud\mu(x)}{\lambda(c_B,d(x,c_B))}
  \leq\sum_{1\leq k<\log_2(4r_{B_1}/r_B)}\int_{2^k B\setminus 2^{k-1}B}
    \frac{\ud\mu(x)}{\lambda(c_B,d(x,c_B))} \\
  &\leq\sum_{1\leq k<\log_2(4r_{B_1}/r_B)}\frac{\mu(c_B,2^k r_B)}{\lambda(c_B,2^{k-1}r_B)}
  \leq\sum_{1\leq k<\log_2(4r_{B_1}/r_B)}C_{\lambda}\leq C_{\lambda}\log_2\Big(\frac{4r_{B_1}}{r_B}\Big).
\end{split}
\end{equation*}
\end{remark}

\begin{lemma}
$\RBMO(\mu)$ is a Banach space.
\end{lemma}

\begin{proof}
One routinely checks that $\RBMO(\mu)$ is a linear space, and $\Norm{\cdot}{\RBMO}$ is a norm when any two functions, whose difference is $\mu$-a.e. equal to a constant, are identified. To prove completeness, first fix a reference ball $B_0$ and replace each function $f^k\in\RBMO(\mu)$, where $\sum_{k=1}^{\infty}\Norm{f^k}{\RBMO}<\infty$, by the function $f^k-f^k_{B_0}$ from the same equivalence class. Also replace the constants $f^k_B$ by $f^k_B-f^k_{B_0}$. Keep denoting these new functions by $f^k$, so that now $f^k_{B_0}=0$. From \eqref{eq:regularity} it follows that $\abs{f^k_B}\leq c(B)\Norm{f^k}{\RBMO}$ for every ball $B$, so in particular the series $\sum_{k=1}^{\infty}f^k_B$ converges to a number $f_B$ for each ball $B$. Using these numbers in the definition of the $\RBMO(\mu)$ space, it is easy to check that $\sum_{k=1}^{\infty}f_k$ converges $\mu$-a.e. and in the norm of $\RBMO(\mu)$ to a function $f$ with $\Norm{f}{\RBMO}\leq\sum_{k=1}^{\infty}\Norm{f^k}{\RBMO}$.
\end{proof}

\begin{lemma}\label{lem:varrho}
The $\RBMO(\mu)$ space is independent of the choice of the parameter $\varrho>1$.
\end{lemma}

\begin{proof}
Denote the $\RBMO(\mu)$ space with parameter $\varrho$ temporarily by $\RBMO_{\varrho}(\mu)$, and let $\varrho>\sigma>1$. It is obvious that $\RBMO_{\sigma}(\mu)\subseteq\RBMO_{\varrho}(\mu)$, where the inclusion map has norm at most $1$, so only the converse direction requires proof.

Let $\delta:=(\sigma-1)/\varrho$ and consider a fixed ball $B_0$. Then there exist balls $B_i=B(x_i,\delta r)$, $x_i\in B_0$ and $i\in I$, which cover $B_0$, where $\abs{I}\leq N\delta^{-n}$. Moreover, $\varrho B_i=B(x_i,\delta\varrho r)\subseteq B(x_0,\sigma r)=\sigma B_0$, since $r+\delta\varrho r=\sigma r$. By Remark~\ref{rem:regUpperBound}, it follows that
\begin{equation*}
\begin{split}
  \abs{f_{B_i}-f_{B_0}}
  \leq\abs{f_{B_i}-f_{\sigma B_0}}+\abs{f_{\sigma B_0}-f_{B_0}}
  &\leq \Norm{f}{\RBMO_{\varrho}}\Big(2+C_{\lambda}\log_2\frac{4r(\sigma B_0)}{r(B_i)}
    +C_{\lambda}\log_2\frac{4r(\sigma B_0)}{r(B_0)}\Big) \\
  &\leq c(\sigma,\varrho)C_{\lambda}\Norm{f}{\RBMO_{\varrho}}.
\end{split}
\end{equation*}
Thus
\begin{equation*}
\begin{split}
  \int_{B_0}\abs{f-f_{B_0}}\ud\mu
  &\leq\sum_{i\in I}\int_{B_i}\abs{f-f_{B_0}}\ud\mu
  \leq\sum_{i\in I}\Big\{\int_{B_i}\abs{f-f_{B_i}}\ud\mu+\mu(B_i)\abs{f_{B_i}-f_{B_0}}\Big\} \\
  &\leq\sum_{i\in I}C\Norm{f}{\RBMO_{\varrho}}\mu(\varrho B_i)
  \leq C\Norm{f}{\RBMO_{\varrho}}\mu(\sigma B_0)\sum_{i\in I}1
  \leq C\Norm{f}{\RBMO_{\varrho}}\mu(\sigma B_0).
\end{split}
\end{equation*}
Hence $\Norm{f}{\RBMO_{\sigma}}\leq C\Norm{f}{\RBMO_{\varrho}}$, and the same numbers $f_B$ work in the definition of both spaces.
\end{proof}

In spaces of homogeneous type, the new BMO space reduces to the classical one:

\begin{proposition}\label{prop:doublingBMO}
If $\mu$ is a doubling measure and $\lambda(x,r)=\mu(B(x,r))$, then $\RBMO(\mu)=\BMO(\mu)$ with equivalent norms.
\end{proposition}

\begin{proof}
If $\mu$ is doubling, then \eqref{eq:BMOrho} is equivalent to the usual BMO condition, and if this condition holds for some $f_B$, it also holds with $f_B=\ave{f}_B$. Hence it remains to investigate the other condition \eqref{eq:regularity} in this case. It will be shown that, in fact,
\begin{equation}\label{eq:regDoubling}
  \abs{\ave{f}_B-\ave{f}_{B_1}}\leq
  C\Norm{f}{\BMO}\Big(1+\log_2\frac{\mu(B_1)}{\mu(B)}\Big)
  \leq C\Norm{f}{\BMO}\Big\{1+\int_{2B_1\setminus B}\frac{\ud\mu(x)}{\mu(B(c_B,d(x,c_B)))}\Big\},
\end{equation}
which proves the assertion.

For $B\subset B_1$, define inductively $B^0:=B$ and $B^i$ to be the smallest $2^k B^{i-1}$, $k\in\N$, with $\mu(2^k B^{i-1})>2\mu(B^{i-1})$; hence $\mu(2^k B^{i-1})\leq C_{\mu}\mu(2^{k-1} B^{i-1})\leq 2C_{\mu}\mu(B^{i-1})$. Let $i_0$ be the first index so that $B^{i_0}\not\subseteq 2B_1$. Then $r(B^{i_0})>r(B_1)$, hence $B_1\subseteq 2B^{i_0}\subseteq B^{i_0+1}$, and therefore
\begin{equation*}
  \mu(B_1)\leq \mu(B^{i_0+1})\leq 2C_{\mu}\mu(B^{i_0})\leq 4C_{\mu}^2\mu(B^{i_0-1})
  \leq 4C_{\mu}^2\mu(2B_1)\leq 4C_{\mu}^3\mu(B_1).
\end{equation*}
Moreover,
$
  2^{i_0}\mu(B)\leq\mu(B^{i_0})\leq (2C_{\mu})^{i_0}\mu(B),
$
and combining these two chains of inequalities,
\begin{equation*}
  2^{i_0-1}C_{\mu}^{-2}\leq\frac{\mu(B_1)}{\mu(B)}\leq (2C_{\mu})^{i_0+1}.
\end{equation*}

Then
\begin{equation*}
\begin{split}
  \abs{\ave{f}_B-\ave{f}_{B_1}}
  &\leq\sum_{i=1}^{i_0+1}\abs{\ave{f}_{B^i}-\ave{f}_{B^{i-1}}}+\abs{\ave{f}_{B^{i_0+1}}-\ave{f}_{B_1}} \\
  &\leq\sum_{i=1}^{i_0+1}\fint_{B^{i-1}}\abs{f-\ave{f}_{B^i}}\ud\mu
    +\fint_{B_1}\abs{f-\ave{f}_{B^{i_0+1}}}\ud\mu \\
  &\leq\sum_{i=1}^{i_0+1}\frac{\mu(B^i)}{\mu(B^{i-1})}\fint_{B^i}\abs{f-\ave{f}_{B^i}}\ud\mu
    +\frac{\mu(B^{i_0+1})}{\mu(B_1)}\fint_{B^{i_0+1}}\abs{f-\ave{f}_{B^{i_0+1}}}\ud\mu \\
  &\leq\sum_{i=1}^{i_0+1}2C_{\mu}\Norm{f}{\BMO}+4C_{\mu}^3\Norm{f}{\BMO}
   \leq C(1+i_0)\Norm{f}{\BMO} \\
  &\leq C\Big(1+\log_2\frac{\mu(B_1)}{\mu(B)}\Big)\Norm{f}{\BMO}.
\end{split}
\end{equation*}
On the other hand, the quantity on the right of \eqref{eq:regularity} can be minorized by
\begin{equation*}
\begin{split}
  &\int_{2B_1\setminus B}\frac{\ud\mu(x)}{\mu(c_B,d(x,c_B))}
  \geq\sum_{i=1}^{i_0-1}\int_{B^i\setminus B^{i-1}}\frac{\ud\mu(x)}{\mu(c_B,d(x,c_B))} \\
  &\geq\sum_{i=1}^{i_0-1}\frac{\mu(B^i\setminus B^{i-1})}{\mu(B^{i-1})}
  \geq\sum_{i=1}^{i_0-1}=i_0-1\geq c\log_2\frac{\mu(B_1)}{\mu(B)}.
\end{split}
\end{equation*}
This completes the proof of \eqref{eq:regDoubling}.
\end{proof}

\section{RBMO and doubling balls}\label{sec:RBMOaux}

Let $(X,d,\mu)$ be geometrically doubling and upper doubling, and let some $\alpha,\beta\geq 2$ be fixed so that the conclusions of Lemmas~\ref{lem:largeDoubling} and \ref{lem:smallDoubling} are valid. Then for every ball $B$, denote by $B'$ the smallest $\alpha^j B$ ($j\in\N$) which is $(\alpha,\beta)$-doubling.

\begin{lemma}\label{lem:doublingAncestor}
For $f\in\RBMO(\mu)$, there holds $\abs{f_B-f_{B'}}\leq C\Norm{f}{\RBMO}$.
\end{lemma}

\begin{proof}
Denote $\gamma:=C_{\lambda}^{\log_2\alpha}$ so that $\beta>\gamma$ by assumption and
\begin{equation*}
  \lambda(x,\alpha^k r)=\lambda(x,2^{k\log_2\alpha}r)\leq C_{\lambda}^{k\log_2\alpha+1}\lambda(x,r)
  =C_{\lambda}\gamma^k\lambda(x,r).
\end{equation*}
Let $B'=\alpha^j B$. Then
\begin{equation*}
\begin{split}
  &\int_{2B'\setminus B}\frac{\ud\mu(x)}{\lambda(c_B,d(x,c_B))}
  \leq\int_{2B'\setminus B'}+\sum_{i=1}^j\int_{\alpha^i B\setminus \alpha^{i-1}B}\ldots \\
  &\leq\frac{\mu(2B')}{\lambda(c_B,\alpha^j r_B)}+\sum_{i=1}^j\frac{\mu(\alpha^i B)}{\lambda(c_B,\alpha^{i-1}r_B)} \\
  &\leq C_{\lambda}+\sum_{i=1}^jC_{\lambda}\frac{\beta^{i-j}\mu(\alpha^j B)}{\gamma^{j-i+1}\lambda(c_B,\alpha^j r_B)}
   \leq C_{\lambda}\Big\{1+\gamma\sum_{i=1}^j\big(\frac{\gamma}{\beta}\big)^{j-i}\Big\}\leq C_{\lambda}c(\beta,\gamma),
\end{split}
\end{equation*}
where the doubling property of $\lambda$ and the non-doubling property of the balls $\alpha^i B$, $i<j$, were used.
\end{proof}

\begin{lemma}\label{lem:neighbours}
For $f\in\RBMO(\mu)$, there holds $\abs{f_{B_1}-f_{B_2}}\leq C\Norm{f}{\RBMO}$ whenever
\begin{equation*}
  d(c(B_1),c(B_2))\leq C_1 \max\{r(B_1),r(B_2)\}\leq C_2\min\{r(B_1),r(B_2)\}.
\end{equation*}
\end{lemma}

\begin{proof}
In the situation of the lemma, there holds $B_1\cup B_2\subseteq m B_1$ and $2m B_1\subseteq MB_2$ for constants $m,M$. Then
\begin{equation*}
  \abs{f_{B_1}-f_{B_2}}\leq\abs{f_{B_1}-mf_{B_1}}+\abs{mf_{B_1}-f_{B_2}},
\end{equation*}
and the second term is bounded by $C\Norm{f}{\RBMO}$ times
\begin{equation*}
\begin{split}
  \int_{2mB_1\setminus B_2}\frac{\ud\mu(x)}{\lambda(c(B_2),d(x,c(B_2)))}
  &\leq \frac{\mu(2m B_1)}{\lambda(c(B_2),r(B_2))}
  \leq \frac{\mu(M B_2)}{\lambda(c(B_2),r(B_2))} \\
  &\leq \frac{\lambda(c(B_2),M r(B_2))}{\lambda(c(B_2),r(B_2))}
  \leq C_{\lambda}^{\log_2 M+1}.
\end{split}
\end{equation*}
The estimate for $\abs{f_{B_1}-mf_{B_1}}$ is similar and slightly simpler, since the second step above is unnecessary then.
\end{proof}

\begin{lemma}\label{lem:aveVsConst}
For $f\in\RBMO(\mu)$ and every $(\alpha,\beta)$-doubling ball $B$, there holds
\begin{equation*}
  \abs{\ave{f}_B-f_B}\leq C\Norm{f}{\RBMO}.
\end{equation*}
\end{lemma}

\begin{proof}
\begin{equation*}
  \abs{\ave{f}_B-f_B}
  \leq\fint_B\abs{f-f_B}\ud\mu
  \leq\frac{\mu(\alpha B)}{\mu(B)}\Norm{f}{\RBMO_{\alpha}}\leq\beta\Norm{f}{\RBMO_{\alpha}}.
\end{equation*}
\end{proof}

\section{The John--Nirenberg inequality}\label{sec:JN}

Everything is now prepared for the main result, which is a simultaneous generalisation of the John--Nirenberg inequalities in Tolsa's $\RBMO(\mu)$ space on $\R^n$, and in the classical $\BMO(\mu)$ space on abstract homogeneous spaces (thanks to Proposition~\ref{prop:doublingBMO}). In the first setting, the inequality is due to Tolsa \cite{Tolsa:RBMO}, and reproven in \cite{NTV:Tb}. In the latter, already Coifman and Weiss \cite[p.~594, footnote]{CW:Hardy} pointed out that John and Nirenberg's proof ``can be adapted to spaces of homogeneous type,'' and explicit proofs can be found in \cite{Buckley,MacP,MMNO}.

\begin{proposition}
Let $(X,d,\mu)$ be geometrically doubling and upper doubling.
For every $\varrho>1$, there is a constant $c$ so that,
for every $f\in\RBMO(\mu)$ and every ball $B_0=B(x_0,r)$,
\begin{equation*}
  \mu(x\in B_0:\abs{f(x)-f_{B_0}}>t)\leq 2\mu(\varrho B_0)e^{-ct/\Norm{f}{\RBMO}}.
\end{equation*}
\end{proposition}

\begin{proof}
Let $\alpha:=5\varrho$, let $\beta$ be large enough as required in Lemmas~\ref{lem:largeDoubling} and \ref{lem:smallDoubling}, and let $L$ be a large constant to be chosen. For every $x\in B_0$, let $B_x'$ be the maximal $(\alpha,\beta)$-doubling ball of the form $B_x'=B(x,\alpha^{-i}r)$, $i\in\N$, such that
\begin{equation*}
  B_x'\subseteq\sqrt{\varrho}B_0,\quad\text{and}\quad\abs{f_{B_x'}-f_{B_0}}>L,
\end{equation*}
if any exist. Note that if $\abs{f(x)-f_{B_0}}>2L$, then there exist (by Corollary~\ref{cor:differentiation}) arbitrarily small doubling balls $B=B(x,\alpha^{-i}r)$ such that $\abs{\ave{f}_B-f_{B_0}}>2L$, and hence
\begin{equation*}
  \abs{f_B-f_{B_0}}> 2L-\abs{\ave{f}_B-f_B}>L,
\end{equation*}
provided that $L\geq C\Norm{f}{\RBMO}$ (using Lemma~\ref{lem:aveVsConst}). Thus for all $x\in B_0$ with $\abs{f(x)-f_{B_0}}>2L$, a ball $B_x'$ will be found. Observe further that
\begin{equation}\label{eq:aveLarge}
  \fint_{B_x'}\abs{f-f_{B_0}}\ud\mu
  \geq\abs{\ave{f}_{B_x'}-f_{B_0}}>L-\abs{\ave{f}_{B_x'}-f_{B_x'}}\geq L-C\Norm{f}{\RBMO}>L/2
\end{equation}
by Lemma~\ref{lem:aveVsConst}, again provided that $L\geq C\Norm{f}{\RBMO}$.

From the maximality of $B_x'$ it follows that $B_x'':=(\alpha B_x')'$ (the minimal $(\alpha,\beta)$-doubling ball of the form $\alpha^iB_x'$, $i\in\Z_+$) satisfies
\begin{equation*}
  B_x''\not\subseteq\sqrt{\varrho}B_0,\quad\text{or}\quad\abs{f_{B_x''}-f_{B_0}}\leq L,
\end{equation*}
In the first case, let $\alpha^i B_x'$, $i\in\Z_+$ be the smallest expansion of $B_x'$ with $\alpha^i B_x'\not\subseteq\sqrt{\varrho}B_0$, so that $r(\alpha^i B_x')\eqsim r(B_0)$ and $B_x''=(\alpha^i B_x')'$. Hence
\begin{equation*}
  \abs{f_{B_x''}-f_{B_0}}\leq\abs{f_{(\alpha^i B_x')'}-f_{\alpha^i B_x'}}
   +\abs{f_{\alpha^i B_x'}-f_{B_0}}\leq C\Norm{f}{\RBMO}
\end{equation*}
by Lemmas~\ref{lem:doublingAncestor} and \ref{lem:neighbours}. But this means that in fact $\abs{f_{B_x''}-f_{B_0}}\leq L$ in any case, provided that $L\geq C\Norm{f}{\RBMO}$. Hence
\begin{equation*}
  L<\abs{f_{B_x'}-f_{B_0}}\leq\abs{f_{B_x'}-f_{B_x''}}+\abs{f_{B_x''}-f_{B_0}}
   \leq C\Norm{f}{\RBMO}+L\leq 3L/2,
\end{equation*}
again provided that $L\geq C\Norm{f}{\RBMO}$.
 
Among the balls $B_x'$, one now chooses disjoint $B_i$, $i\in I$, so that the expanded balls $5B_i$ cover all the original $B_x'$. This is again an application of the basic covering theorem \cite[Theorem 1.2]{Heinonen:book}. If $x\in 5B_i$ and $\abs{f(x)-f_{B_0}}>nL$, then
\begin{equation*}
  \abs{f(x)-f_{5B_i}}
  \geq\abs{f(x)-f_{B_0}}-\abs{f_{B_0}-f_{B_i}}-\abs{f_{B_i}-f_{5B_i}}
  >nL-3L/2-C\Norm{f}{\RBMO}\geq (n-2)L
\end{equation*}
if $L\geq C\Norm{f}{\RBMO}$. For $n\geq 2$, it thus follows that
\begin{equation*}
\begin{split}
  \{x\in B_0:\abs{f(x)-f_{B_0}}>nL\}
  &\subseteq\bigcup_{\ontop{x\in B_0:}{\abs{f(x)-f_{B_0}}>nL}}\{y\in B_x':\abs{f(y)-f_{B_0}}>nL\} \\
  &\subseteq\bigcup_{i\in I}\{y\in 5B_i:\abs{f(y)-f_{5B_i}}>(n-2)L\}.
\end{split}
\end{equation*}

Using \eqref{eq:aveLarge} and the fact that the balls $B_i=B_{x_i}'$ are $(\alpha,\beta)$-doubling, disjoint, and contained in $\sqrt{\varrho}B_0$, it follows that
\begin{equation*}
\begin{split}
  \sum_{i\in I}\mu(\varrho\cdot 5B_i)
  &=\sum_{i\in I}\mu(\alpha B_i)
  \leq\beta\sum_{i\in I}\mu(B_i)
  \leq\frac{C}{L}\sum_{i\in I}\int_{B_i}\abs{f-f_{B_0}}\ud\mu \\
  &\leq\frac{C}{L}\int_{\sqrt{\varrho}B_0}\abs{f-f_{B_0}}\ud\mu \\
  &\leq\frac{C}{L}\Big(\int_{\sqrt{\varrho}B_0}\abs{f-f_{\sqrt{\varrho}B_0}}\ud\mu
    +\mu(\sqrt{\varrho}B_0)\abs{f_{\sqrt{\varrho}B_0}-f_{B_0}}\Big) \\
  &\leq\frac{C}{L}\Big(\mu(\sqrt{\varrho}\cdot\sqrt{\varrho}B_0)\Norm{f}{\RBMO_{\sqrt{\varrho}}}
     +\mu(\sqrt{\varrho}B_0)\Norm{f}{\RBMO_{\sqrt{\varrho}}}\Big) \\
  &\leq\frac{C}{L}\Norm{f}{\RBMO}\mu(\varrho B_0)\leq\frac{1}{2}\mu(\varrho B_0),
\end{split}
\end{equation*}
given that $L\geq C\Norm{f}{\RBMO}$.

Writing $B^i:=5B_i$, the above results can be summarised as
\begin{equation*}
\begin{split}
  \{x\in B_0:\abs{f(x)-f_{B_0}}>nL\}
  &\subseteq\bigcup_{i\in I}\{x\in B^i:\abs{f(x)-f_{B^i}}>(n-2)L\}, \\
  \sum_{i\in I}\mu(\varrho B^i) &\leq\frac{1}{2}\mu(\varrho B_0).
\end{split}
\end{equation*}
This contains the essence of the matter, for now one can iterate with the balls $B^i$ in place of $B_0$, to the result that
\begin{equation*}
\begin{split}
  \{x\in B_0:\abs{f(x)-f_{B_0}}>2nL\}
  &\subseteq\bigcup_{i_1}\{x\in B^{i_1}:\abs{f(x)-f_{B^{i_1}}}>2(n-1)L\}\\
  &\subseteq\bigcup_{i_1,i_2}\{x\in B^{i_1,i_2}:\abs{f(x)-f_{B^{i_1,i_2}}}>2(n-2)L\}\subseteq\ldots\\
  &\subseteq\bigcup_{i_1,i_2\ldots,i_n}\{x\in B^{i_1,i_2,\ldots,i_n}:\abs{f(x)-f_{B^{i_1,i_2\ldots,i_n}}}>0\},
\end{split}
\end{equation*}
and then
\begin{equation*}
\begin{split}
  \mu(x\in B_0:\abs{f(x)-f_{B_0}}>2nL)
  &\leq \sum_{i_1,\ldots,i_{n-1},i_n}\mu(B^{i_1,\ldots,i_{n-1},i_n}) \\
  &\leq \sum_{i_1,\ldots,i_{n-1}}\sum_{i_n}\mu(\varrho B^{i_1,\ldots,i_{n-1},i_n}) \\
  &\leq \sum_{i_1,\ldots,i_{n-1}}\frac{1}{2}\mu(\varrho B^{i_1,\ldots,i_{n-1}})\leq\ldots
   \leq \frac{1}{2^n}\mu(\varrho B_0).
\end{split}
\end{equation*}

Recall that one can take $L=C\Norm{f}{\RBMO}$, and choose $n\in\N$ so that $2nL\leq t<2(n+1)L$. Thus
\begin{equation*}
\begin{split}
  \mu(x\in B_0:\abs{f(x)-f_{B_0}}>t)
  &\leq \mu(x\in B_0:\abs{f(x)-f_{B_0}}>2nL) \\
  &\leq 2^{-n}\mu(\varrho B_0) \leq 2^{-(2L)^{-1}t+1}\mu(\varrho B_0)
   = 2e^{-ct/\Norm{f}{\RBMO}}\mu(\varrho B_0),
\end{split}  
\end{equation*}
and this completes the proof.
\end{proof}

The familiar corollary follows in the usual way, and is left as an exercise:

\begin{corollary}
Let $(X,d,\mu)$ be geometrically doubling and upper doubling.
For every $\varrho>1$ and $p\in[1,\infty)$, there is a constant $C$ so that,
for every $f\in\RBMO(\mu)$ and every ball $B_0$,
\begin{equation*}
  \Big(\frac{1}{\mu(\varrho B_0)}\int_{B_0}\abs{f-f_{B_0}}^p\ud\mu\Big)^{1/p}
  \leq C\Norm{f}{\RBMO}.
\end{equation*}
\end{corollary}


\bibliographystyle{plain}
\bibliography{metric}

\end{document}